\documentclass{amsart}
\usepackage{a4wide,latexsym,amssymb,amsthm,amsmath,color}
\usepackage{mathtools}
\theoremstyle{plain}

\newtheorem{theorem}{Theorem}
\newtheorem{lemma}{Lemma}

\newtheorem{definition}{Definition}

\theoremstyle{remark}

\def\R{\mathbb{R}}

\newcommand{\nks}{\ensuremath{\mathbb{S}^3 \times
\mathbb{S}^3}}

 \numberwithin{equation}{section}

\title[$H$-umbilical Lagrangian submanifolds of the nearly Kähler $\mathbb{S}^3\times\mathbb{S}^3$]
{$H$-umbilical Lagrangian submanifolds \\ of the nearly Kähler $\mathbb{S}^3\times\mathbb{S}^3$}

\author[M. Anti\'c]{Miroslava Anti\'c}
\address{University of Belgrade, Faculty of Mathematics, Studentski trg 16, pb. 550, 11000 Belgrade, Serbia}
\email{mira@matf.bg.ac.rs}

\author[M. Moruz]{Marilena Moruz}
\address{KU Leuven, Department of Mathematics, Celestijnenlaan 200B -- Box 2400, 3001 Leuven, Belgium}
\email{marilena.moruz@kuleuven.be}

\author[J. Van der Veken]{Joeri Van der Veken}
\address{KU Leuven, Department of Mathematics, Celestijnenlaan 200B -- Box 2400, 3001 Leuven, Belgium} 
\email{joeri.vanderveken@kuleuven.be}

\thanks{M. Moruz is a postdoctoral fellow of the Research Foundation -- Flanders (FWO). J. Van der Veken is supported by the Excellence of Science project G0H4518N of the Belgian government, project 3E160361 of the KU Leuven Research Fund and collaboration project G0F2319N of the Research Foundation -- Flanders (FWO) and the National Natural Science Foundation of China (NSFC)}

\begin{document}



\begin{abstract}
$H$-umbilicity was introduced as an analogue of total umbilicity for Lagrangian submanifolds since, in some relevant cases, totally umbilical Lagrangian submanifolds are automatically totally geodesic. In this paper, we show that in the homogeneous nearly K\"ahler $\mathbb{S}^3\times\mathbb{S}^3$, also $H$-umbilical Lagrangian submanifolds are automatically totally geodesic.
\end{abstract}

\maketitle

\section{Introduction}

Gray and Hervella \cite{GH} have distinguished sixteen classes of almost Hermitian manifolds, out of which the K\"ahler manifolds represent one of the most important classes. K\"ahler manifolds are defined as the almost Hermitian manifolds for which the almost complex structure $J$ is parallel: $\tilde \nabla J = 0$, where $\tilde \nabla$ is the Levi Civita connection. If this condition is relaxed to the skew-symmetry of the tensor $\tilde \nabla J$, then the manifold is called nearly K\"ahler. It is known that K\"ahler manifolds admit a complex, Riemannian and symplectic structure that are all compatible with each other.  On a nearly K\"ahler manifold, the fundamental two-form is not necessarily closed and therefore nearly K\"ahler manifolds are neither complex nor symplectic, unless they are K\"ahler, see for example \cite{thesis}.

The interest in nearly K\"ahler manifolds grew especially because they are examples of geometries with torsion and therefore they have applications in mathematical physics \cite{A}. A very important result is the structure theorem given by Nagy \cite{N}. He showed that strict nearly K\"ahler manifolds are Riemannian products of six-dimensional nearly Kähler manifolds, some homogenous nearly K\"ahler spaces and  twistor spaces over quaternionic K\"ahler manifolds with positive scalar curvature, endowed with the canonical nearly K\"ahler metric.

Butruille showed in \cite{B} that the only homogeneous $6$-dimensional nearly K\"ahler manifolds are the nearly K\"ahler $6$-sphere $\mathbb{S}^6$, the manifold $\nks$, the projective space $\mathbb{C}P^3$ and the flag manifold $SU(3)/U(1) \times U(1)$, where the last three are not endowed with their ``standard metrics''. All these spaces are compact $3$-symmetric spaces. Furthermore, in 2015, Cort\'es and V\'asquez have discovered the first non-homogeneous (but locally homogeneous) nearly K\"ahler structures in \cite{CV}, while in 2017 the first complete non-homogeneous nearly K\"ahler structures were constructed on $\mathbb{S}^6$ and $\nks$ by Foscolo and Haskins in \cite{FH}.

In the present paper, we are interested in the study of the homogeneous nearly K\"ahler $\nks$ from the point of view of its submanifolds. This aligns with the interest of other authors who have recently contributed with results about various properties of different types of submanifolds of $\nks$, such as almost complex surfaces, hypersurfaces and CR or Lagrangian submanifolds (see, for instance \cite{Antic, Bektas, Bektas1, Dioos1, Dioos2, Hu, MS}). The study of Lagrangian submanifolds originates from symplectic geometry and classical mechanics. Sch\"afer and Smozcyk \cite{SS} proved that Lagrangian submanifolds of nearly K\"ahler manifolds of dimension six and twistor spaces are always minimal and orientable. 

Recall that a submanifold $M$ of dimension $n \geq 2$ of a manifold $\tilde M$ is said to be totally umbilical if its second fundamental form $h$ is given by 
\begin{equation} \label{totallyumbilical}
h(X,Y)=g(X,Y)\xi
\end{equation}
for all $X$ and $Y$ tangent to $M$, where $g$ is the metric and $\xi$ is a normal vector field along the submanifold, not depending on $X$ and $Y$. It follows directly from \eqref{totallyumbilical} that $\xi = H$, the mean curvature vector field of the immersion. In the case that the ambient space carries an almost complex structure $J$, we define the cubic form acting on tangent vectors to $M$ by $(X,Y,Z) \mapsto g(h(X,Y),JZ)$ and the following result is easy to prove.
\begin{lemma} \label{lem:tu}
Let $M$ be a totally umbilical Lagrangian submanifold of an almost Hermitian manifold $(\tilde M,g,J)$. If the cubic form of the immersion is totally symmetric in its three arguments, then the submanifold is totally geodesic, i.e., the second fundamental form vanishes identically.
\end{lemma}

Symmetry of the cubic form holds for important classes of Lagrangian submanifolds, such as all Lagrangian submanifolds of all K\"ahler manifolds and nearly K\"ahler manifolds, see for example \cite{thesis}. This lead Chen to introduce the concept of $H$-umbilicity as an alternative for total umbilicity for Lagrangian submanifolds in \cite{Chen}. To motivate the definition, we first remark that, in the case of a Lagrangian submanifold, \eqref{totallyumbilical} can be expressed as follows: in a neighborhood of any point of $M$, there exist a local orthonormal frame $\{\mathbf{U}_1,\ldots,\mathbf{U}_n\}$ on $M$ and a local function $\lambda$ such that the second fundamental form is given by
\begin{align*} 
h(\mathbf{U}_1,\mathbf{U}_1) = \lambda J \mathbf{U}_1, \quad 
h(\mathbf{U}_1,\mathbf{U}_i) = 0, \quad 
h(\mathbf{U}_i,\mathbf{U}_j) = \lambda \delta_{ij} J \mathbf{U}_1
\end{align*}
for $i,j \in \{2,\ldots,n\}$.

\begin{definition} \cite{Chen} \label{def:H-umbilical}
Let $M$ be a Lagrangian submanifold of dimension $n \geq 2$ of an almost Hermitian manifold $(\tilde M,g,J)$. We say that $M$ is \emph{$H$-umbilical} if, in the neighborhood of any point, there exist a local orthonormal frame $\{\mathbf{U}_1,\ldots,\mathbf{U}_n\}$ on $M$ and local functions $\lambda$ and $\mu$ such that the second fundamental form $h$ is given by
\begin{align} \label{humbilical}
h(\mathbf{U}_1,\mathbf{U}_1) = \lambda J \mathbf{U}_1, \quad 
h(\mathbf{U}_1,\mathbf{U}_i) = \mu J \mathbf{U}_i, \quad 
h(\mathbf{U}_i,\mathbf{U}_j) = \mu \delta_{ij} J \mathbf{U}_1
\end{align}
for $i,j \in \{2,\ldots,n\}$.
\end{definition}
 
While there are plenty of examples and classification results of $H$-umbilical Lagrangian submanifolds in the K\"ahler setting --see for example \cite{Chen} and \cite{Chen2} for the classification in complex space forms-- we prove in this paper that no new examples occur in the nearly K\"ahler $\nks$. 

\begin{theorem}\label{Th}
An $H$-umbilical Lagrangian submanifold of $\nks$ is totally geodesic.
\end{theorem}

Note that a complete classification of totally geodesic Lagrangian submanifolds of $\nks$ is given in \cite{Zhang}.

The remaining part of this paper is organized as follows: in Section 2 we introduce the necessary basics about $\nks$ and its Lagrangian submanifolds and in Section 3 we will prove Theorem \ref{Th}.

\section{Preliminaries}

We will briefly present the homogeneous nearly K\"ahler structure on $\nks$ as well as the basic equations for Lagrangian submanifolds of $\nks$, as they have been introduced in \cite{complexsurfaces} and \cite{Zhang}.

Looking at the $3$-sphere as the set of all unit quaternions, it is not hard to see that the tangent space at $p \in \mathbb S^3$ can be identified with $T_p\mathbb S^3 = \{ p\alpha \, | \, \alpha \mbox{ is an imaginary quaternion}\}$, where $p\alpha$ denotes the quaternionic multiplication of the unit quaternion $p$ and the imaginary quaternion~$\alpha$. A tangent vector to $\nks$ at a point $(p,q) \in \nks$ can hence be written as $( p\alpha,q\beta)$ for imaginary quaternions $\alpha$ and $\beta$.

Using this notation, the almost complex structure $J$ on the nearly K\"ahler $\mathbb{S}^3\times\mathbb{S}^3$ is defined by 
\begin{equation}
J_{(p,q)}(p\alpha,q\beta) := \frac{1}{\sqrt{3}} (p(2\beta-\alpha), q(\beta-2\alpha)).
\end{equation}
A Hermitian metric which is compatible with this almost complex structure can now be built starting from the product of the round metrics of curvature $1$ in a standard way: 
\begin{equation}\label{g}
\begin{aligned}
g_{(p,q)}((p\alpha,q\beta),(p\alpha',q\beta'))
:=& \ \frac{1}{2} \left( \langle (p\alpha,q\beta),(p\alpha',q\beta') \rangle + \langle J_{(p,q)}(p\alpha,q\beta),J_{(p,q)}(p\alpha',q\beta') \rangle \right)\\
=& \ \frac{4}{3} \left( \langle \alpha,\alpha' \rangle + \langle \beta,\beta' \rangle \right) - \frac{2}{3} \left( \langle \alpha,\beta' \rangle + \langle \alpha',\beta \rangle \right),
\end{aligned}
\end{equation}
where $\langle \cdot \, ,\cdot \rangle $ stands for the product of the round metrics of curvature $1$ (or, equivalently, for the Euclidean metric on $\R^8$) in the first line and for the Euclidean metric on $\R^3$ in the second line. Following \cite{complexsurfaces}, we define one more tensor field on $\nks$, namely
\begin{equation}\label{defP}
P_{(p,q)}(p\alpha, q\beta) := (p\beta,q\alpha).
\end{equation}
Note that $P$ is an almost product structure since it is involutive and symmetric with respect to~$g$. Moreover, it anti-commutes with the almost complex structure: $JP=-PJ$.

Denote by $\tilde\nabla$ the Levi-Civita connection on $\mathbb{S}^3\times \mathbb{S}^3$ with respect to the metric $g$ and let $G:=\tilde\nabla J$. Since $G$ is skew-symmetric, the Hermitian manifold $(\nks,g,J)$ is indeed nearly K\"ahler. 

Now let $M$ be a Lagrangian submanifold of $\nks$. This means that $J$ maps tangent vectors to $M$ to normal vectors to $M$ and vice versa, which, in particular, implies that $M$ is $3$-dimensional. The formulas of Gauss and Weingarten, which hold for general submanifolds, state respectively
\begin{align*}
& \tilde{\nabla}_XY = \nabla_XY + h(X,Y),\\
& \tilde{\nabla}_X\xi = -S_{\xi}X + \nabla^{\perp}_X\xi
\end{align*}
for vector fields $X$ and $Y$ tangent to $M$ and a vector field $\xi$ normal to $M$. Here, $\nabla$ is the Levi-Civita connection on $M$ with respect to the metric induced by $g$, $h$ is the second fundamental form of the immersion, $S_{\xi}$ is the shape operator with respect to the normal vector field $\xi$ and $\nabla^{\perp}$ is the normal connection. In the case of a Lagrangian submanifold of $\nks$ --or, more generally, a Lagrangian submanifold of any (nearly) K\"ahler manifold-- the following equations describe the relations between some of these geometric data:
\begin{align}
& h(X,Y) = JS_{JX}Y = JS_{JY} X, \label{eq:h_for_nearly_Kahler}\\
& \nabla^{\perp}_XJY = J\nabla_XY + G(X, Y ) \label{eq:nabla_perp_for_nearly_Kahler}
\end{align}
for all vector fields $X$ and $Y$ tangent to the Lagrangian submanifold $M$. In particular, it follows from \eqref{eq:h_for_nearly_Kahler} that the cubic form $(X,Y,Z) \mapsto g(h(X,Y ),JZ)$ on $M$ is totally symmetric.

Since $M$ is Lagrangian, the pull-back of $T(\mathbb{S}^3 \times \mathbb{S}^3)$ to $M$ splits into $TM\oplus JTM$. There are two endomorphisms $A,B\colon TM\to TM$ such that the restriction $P|_{TM}$ of $P$ to the submanifold is given by $PX = AX + JBX$ for all $X\in TM$. It follows immediately from the properties of $P$ that $A$ and $B$ are symmetric, that they commute and satisfy  $A^2+B^2=\mbox{Id}$. Therefore, for each point $p\in M$ there is an orthonormal basis $\{e_1,e_2,e_3\}$ of $T_p M$ and real numbers $\theta_1$, $\theta_2$ and $\theta_3$, determined up to an integer multiple of $\pi$, such that $A e_i = \cos(2\theta_i) e_i$ and $Be_i = \sin(2\theta_i) Je_i$ for $i \in \{1,2,3\}$. We may extend the orthonormal basis $\{e_1, e_2, e_3\}$ at a point to a  differentiable frame $\{E_1,E_2,E_3\}$ on an open neighborhood where the multiplicities of the eigenvalues of $A$ and $B$ are constant. Therefore, there exist a local orthonormal frame $\{E_1,E_2,E_3\}$ and local functions $\theta_1$, $\theta_2$ and $\theta_3$ on an open dense subset of $M$ such that
\begin{equation}\label{diagop}
AE_i=\cos (2\theta_i)  E_i, \quad BE_i=\sin (2\theta_i) E_i
\end{equation}
for $i \in \{1,2,3\}$.

The following lemma is a combination of two lemmas proven in \cite{Dioos2}.

\begin{lemma} \cite{Dioos2} \label{lem:angle_functions}
Let $M$ be a Lagrangian submanifold of $\nks$ and define the local angle functions $\theta_1$, $\theta_2$ and $\theta_3$ as above. Then, the sum of these functions vanishes modulo $\pi$. Moreover, if two of the angle functions are equal modulo $\pi$, then the submanifold is totally geodesic.
\end{lemma}

The next lemma follows from studying $\tilde\nabla P$ and was also proven in \cite{Dioos2}.
\begin{lemma} \cite{Dioos2} \label{lem:formulas_hijk}
Let $M$ be a Lagrangian submanifold of $\nks$ and take a local orthonormal frame $\{E_1,E_2,E_3\}$ and local functions $\theta_1$, $\theta_2$ and $\theta_3$ as above. Denote by $h_{ij}^k$ the components of the second fundamental form of the immersion and by $\omega_{ij}^k$ the components of the Levi-Civita connection of $M$ with respect to $\{E_1,E_2,E_3\}$, i.e.,
$$ h_{ij}^k := g(h(E_i,E_j),JE_k), \qquad \omega_{ij}^k := g(\nabla_{E_i}E_j,E_k) $$ 
for $i,j,k \in \{1,2,3\}$. Then
\begin{align}
& E_i(\theta_j) = -h_{jj}^i, \nonumber \\
& h_{ij}^k \cos(\theta_j-\theta_k) = \left( \frac{1}{2\sqrt 3} \varepsilon_{ij}^k - \omega_{ij}^k \right) \sin(\theta_j-\theta_k) \label{eq:relation_hijk_omegaijk}
\end{align}
for $i,j,k \in \{1,2,3\}$, with $j \neq k$. Here, $\varepsilon_{ij}^k$ is defined by
$$\varepsilon_{ij}^k = \left\{
\begin{array}{rl}
1 & \mbox{if $(ijk)$ is an even permutation of $(123)$,} \\
-1 & \mbox{if $(ijk)$ is an odd permutation of $(123)$,} \\
0 & \mbox{otherwise.}
\end{array}
\right.
$$
\end{lemma}

By changing the orientation of one of the vectors in a local orthonormal frame $\{E_1,E_2,E_3\}$ on a Lagrangian submanifold of $\nks$ if necessary, we way assume that 
\begin{equation}\label{jg}
G(E_i,E_j) = -\frac{1}{\sqrt{3}} \sum_{k=1}^3 \varepsilon_{ij}^k JE_k,
\end{equation}
where $\varepsilon_{ij}^k$ is as in Lemma \ref{lem:formulas_hijk}, see for example \cite{thesis}. Hence, from now on, we will assume this for the frame $\{E_1,E_2,E_3\}$ constructed above.

Finally, we recall the equation of Codazzi for a Lagrangian submanifold $M$ of $\nks$ from \cite{Dioos2}. It states that
\begin{equation}\label{cdz}
\begin{aligned}
(\overline{\nabla}h)(X,Y,Z) & - (\overline{\nabla}h)(Y,X,Z) \\ = & \frac{1}{3}(g(AY,Z)JBX-g(AX,Z)JBY-g(BY,Z)JAX+g(BX,Z)JAY)
\end{aligned}
\end{equation}
for all tangent vector fields $X$, $Y$ and $Z$ on $M$, where the covariant derivative of the second fundamental form is defined as $(\overline{\nabla}h)(X,Y,Z) := \nabla^{\perp}_Xh(Y,Z) -h(\nabla_XY,Z) - h(Y,\nabla_XZ)$.



\section{ Proof of Theorem \ref{Th}}

First, we prove the following lemma.
\begin{lemma}\label{lem:invariant_form_h}
Let $M$ be an $H$-umbilical Lagrangian submanifold of the nearly K\"ahler $\nks$. Then, in a neihborhood of any point, there exists a vector field $V$ on $M$ such that the second fundamental form of $M$ is given by
\begin{equation}\label{formofh}
h(X,Y)=g(V,V)\left(g(Y,V)JX+g(X,V)JY+g(X,Y)JV\right)-5g(X,V)g(Y,V)JV,
\end{equation}
for all vector fields $X$ and $Y$ on $M$ in that neighborhood.
\end{lemma}

\begin{proof}
Let $M$ be a Lagrangian $H$-umbilical submanifold of the nearly K\"ahler $\nks$. By Definition~\ref{def:H-umbilical}, around every point of $M$, there exist a local  orthonormal frame $\{\mathbf{U}_1,\mathbf{U}_2,\mathbf{U}_3\}$ and local functions $\lambda$ and $\mu$ such that 
\begin{equation} \label{eq:H-umbilical_dimension_3}
\begin{aligned}
& h(\mathbf{U}_1,\mathbf{U}_1) = \lambda J\mathbf{U}_1, 
&& h(\mathbf{U}_1,\mathbf{U}_2) = \mu J\mathbf{U}_2, 
&& h(\mathbf{U}_1,\mathbf{U}_3) = \mu J\mathbf{U}_3, \\
& h(\mathbf{U}_2,\mathbf{U}_2) = \mu J\mathbf{U}_1, 
&& h(\mathbf{U}_2,\mathbf{U}_3) = 0, 
&& h(\mathbf{U}_3,\mathbf{U}_3) = \mu J\mathbf{U}_1.
\end{aligned}
\end{equation}
Since Lagrangian submanifolds of $\nks$ are minimal by \cite{SS}, we must have $\lambda=-2\mu$. If we define $V := \mu^{1/3} \mathbf{U}_1$, it is then easy to see that \eqref{eq:H-umbilical_dimension_3} is equivalent to \eqref{formofh}.
\end{proof}

We are now ready to give a proof of Theorem \ref{Th}.
\begin{proof}[Proof of Theorem \ref{Th}.]

Assume that $M$ is an $H$-umbilical Lagrangian submanifold of $\nks$. Let $\{E_1,E_2,E_3\}$ be an orthonormal frame of tangent vector fields and $\theta_1$, $\theta_2$ and $\theta_3$ functions on an open dense subset of $M$ such that \eqref{diagop} and \eqref{jg} hold. If $V$ is a local vector field on $M$ as defined in Lemma \ref{lem:invariant_form_h} and $v_1$, $v_2$ and $v_3$ are local functions such that $V=v_1E_1+v_2E_2+v_3E_3$, then it follows from Lemma \ref{lem:formulas_hijk} and Lemma \ref{lem:invariant_form_h} that
\begin{equation} \label{eq:hijk}
h_{ij}^k = (v_1^2+v_2^2+v_3^2)(v_i \delta_{jk} + v_j \delta_{ki} + v_k \delta_{ij}) - 5 v_i v_j v_k
\end{equation}
for all $i,j,k \in \{1,2,3\}$. 

By substituting \eqref{eq:hijk} into \eqref{eq:relation_hijk_omegaijk}, we obtain the following:
\begin{align}
& \omega_{11}^2 = -v_2(-4v_1^2+v_2^2+v_3^2)\cot(\theta_1-\theta_2),
&& \omega_{11}^3 = -v_3(-4v_1^2+v_2^2+v_3^2)\cot(\theta_1-\theta_3), \nonumber \\
& \omega_{22}^1 = -v_1(v_1^2-4v_2^2+v_3^2 )\cot(\theta_2-\theta_1),
&& \omega_{22}^3 = -v_3(v_1^2-4 v_2^2+v_3^2)\cot(\theta_2-\theta_3), \nonumber \\
& \omega_{33}^1 = -v_1(v_1^2+v_2^2-4 v_3^2)\cot(\theta_3-\theta_1),
&& \omega_{33}^2 = -v_2(v_1^2+v_2^2-4v_3^2)\cot(\theta_3-\theta_2) \label{eq:omegaijk} \\
& \omega_{12}^3 = \frac{1}{2\sqrt 3} + 5 v_1 v_2 v_3\cot(\theta_2-\theta_3),
&& \omega_{23}^1 = \frac{1}{2\sqrt 3} +5 v_1 v_2 v_3 \cot(\theta_3-\theta_1), \nonumber \\
& \omega_{31}^2 = \frac{1}{2\sqrt 3} + 5 v_1 v_2 v_3 \cot(\theta_1-\theta_2). 
&& \nonumber
\end{align}
Note that we have divided by factors of type $\sin(\theta_i-\theta_j)$. If such a factor is identically zero, two of the angle functions are equal modulo $\pi$ and hence $M$ is totally geodesic by Lemma~\ref{lem:angle_functions}, which would finish the proof. Hence, from now on, we assume we are not in this situation and restrict the open and dense subset of $M$ we are working on to the subset of those points at which $\sin(\theta_i-\theta_j) \neq 0$ for different $i,j \in \{1,2,3\}$. Given the symmetry $\omega_{ij}^k = -\omega_{ik}^j$, the equalities in \eqref{eq:omegaijk} completely determine the Levi-Civita connection of $M$.

The rest of the proof will rely on the equation of Codazzi. Evaluating \eqref{cdz} for suitable choices of $X$, $Y$ and $Z$ from $\{E_1,E_2,E_3\}$, in combination with \eqref{eq:nabla_perp_for_nearly_Kahler}, \eqref{diagop}, \eqref{jg}, \eqref{eq:hijk} and \eqref{eq:omegaijk}, yields differential equations for the functions $v_1$, $v_2$ and $v_3$, also involving the functions $\theta_1$, $\theta_2$ and $\theta_3$.

In particular, if we choose $(X,Y,Z)$ equal to $(E_1,E_2,E_1)$, $(E_1,E_2,E_2)$ and $(E_1,E_2,E_3)$ in \eqref{cdz}, we can express the derivatives $E_2(v_1)$, $E_1(v_2)$, $E_2(v_2)$, $E_1(v_3)$ and $E_2(v_3)$ in terms of $v_1$, $v_2$, $v_3$, $\theta_1$, $\theta_2$ and $\theta_3$. However, since these expressions are very long, we do not give them explicitly here. Similarly, by letting $(X,Y,Z)$ be equal to $(E_1,E_3,E_1)$, $(E_1,E_3,E_2)$ and $(E_1,E_3,E_3)$ in \eqref{cdz}, we can solve the resulting system of equations for $E_3(v_1)$, $E_1(v_2)$, $E_3(v_2)$, $E_1(v_3)$ and $E_3(v_3)$. By comparing the expressions for $E_1(v_2)$ and $E_1(v_3)$ from both approaches, we obtain the following system of algebraic equations:
\begin{equation}\label{system1}
\begin{aligned}
& v_1 v_2 \big[ (4 v_1^4 + v_3^4 + 4v_1^2v_2^2 + 12 v_1^2v_3^2 + v_2^2v_3^2) \sin(2 (\theta_1 - \theta_2)) \\
& \hspace{4cm} -(4v_1^4 + 4v_2^4 + 3 v_3^4 + 8v_1^2v_2^2  + 7 v_2^2 v_3^2) \sin(2 (\theta_1 - \theta_3)) \big] = 0, \\
& v_1 v_3 \big[ (4v_1^4 + 3v_2^4  + 4 v_3^4 + 8v_1^2v_3^2 + 7v_2^2v_3^2) \sin(2(\theta_1 - \theta_2)) \\ 
& \hspace{4cm} - (4v_1^4 + v_2^4 + 12v_1^2v_2^2 + 4v_1^2v_3^2 + v_2^2v_3^2) \sin(2(\theta_1-\theta_3)) \big] = 0.
\end{aligned}
\end{equation} 
These equations take the form $v_1 v_2 F_1(v_1,v_2,v_3,\theta_1,\theta_2,\theta_3)=0$ and $v_1 v_3 F_2(v_1,v_2,v_3,\theta_1,\theta_2,\theta_3)=0$ and it therefore suffices to consider the following five cases: Case 1: $v_2=v_3=0$; Case 2: $v_1=0$; Case 3: $v_2=F_2(v_1,v_2,v_3,\theta_1,\theta_2,\theta_3)=0$; Case 4: $v_3=F_1(v_1,v_2,v_3,\theta_1,\theta_2,\theta_3)=0$; Case 5: $F_1(v_1,v_2,v_3,\theta_1,\theta_2,\theta_3)=F_2(v_1,v_2,v_3,\theta_1,\theta_2,\theta_3)=0$.

Before proceeding, we remark that in order to solve for the above mentioned derivatives, and hence to obtain \eqref{system1}, we have assumed that
\begin{equation}\label{ec}
4 v_1^2 - 3 (v_2^2 + v_3^2)\neq 0.
\end{equation}
In fact, $4v_i^2-3(v_j^2+v_k^2)$ is non-zero for at least one permutation $(ijk)$ of $(123)$. If this was not true, we would have $V=0$ and hence $M$ would be totally geodesic by \eqref{formofh}. Therefore, we may indeed assume that $4 v_1^2 - 3 (v_2^2 + v_3^2)$ is non-zero.

\textbf{Case 1: $v_2=v_3=0$.} It is straightforward to see that the equation of Codazzi evaluated for $(X,Y,Z)=(E_1,E_2,E_1)$ implies that $v_1=0$. Hence, $V=0$ and $M$ is totally geodesic.

\textbf{Case 2: $v_1=0$.} The equation of Codazzi evaluated for $(X,Y,Z)=(E_1,E_2,E_1)$ gives three equations, from which we can express  $E_2(v_3)$, $E_1(v_2)$ and $E_2(v_2)$ in terms of $v_1$, $v_2$, $v_3$, $\theta_1$, $\theta_2$, $\theta_3$ and $E_1(v_3)$. Since these expressions are long, we omit giving them explicitly. However, if we replace these derivatives in the equation of Codazzi for $(X,Y,Z)=(E_1,E_2,E_2)$, we obtain 
\begin{align*}
& v_3 (3v_2^4-v_2^2v_3^2+v_3^4 + 15\sqrt{3} \, v_2^3E_1(v_3))=0,\\
& v_2 (3v_2^4-3 v_2^2v_3^2+4v_3^4) + 
3\sqrt{3} (4v_2^4-7v_2^2v_3^2-v_3^4) E_1(v_3)=0.
\end{align*}
These equations for $E_1(v_3)$ are only compatible if $v_2=v_3=0$, which implies that $V=0$ and hence that $M$ is totally geodesic. 

\textbf{Case 3: $v_2 = (v_1^2 + v_3^2) \left[(v_1^2 + v_3^2) \sin(2(\theta_1-\theta_2)) - v_1^2 \sin(2(\theta_1-\theta_3)) \right] = 0$.} The expressions for $E_2(v_3)$ and $E_2(v_1)$ obtained in the beginning of the proof, before equations \eqref{system1}, simplify to 
\begin{align*}
& E_2(v_3) = \frac{v_1 (6v_1^4 + 5v_1^2 v_3^2 + 4v_3^4)}{ 3\sqrt{3} (8 v_1^4 + 6 v_1^2 v_3^2 + 3 v_3^4)}, \\
& E_2(v_1) = -\frac{v_3(10v_1^4 + 8v_1^2 v_3^2 + 3v_3^4)}{ 3\sqrt{3} (8 v_1^4 + 6 v_1^2 v_3^2 + 3 v_3^4)}.
\end{align*}
If we substitute the above expressions in the Codazzi equation for $(X,Y,Z)=(E_1,E_2,E_3)$, we obtain that $v_3(v_1^2+v_3^2)=0$. We consider two cases in order to prove that $M$ is totally geodesic. If $v_1^2+v_3^2=0$, we have $V=0$, since we already have $v_2=0$ in this case. If $v_3=0$, then the equation of Codazzi for $(X,Y,Z)=(E_1,E_2,E_1)$ yields immediately that $v_1=0$ and hence also that $V=0$. In both cases, we conclude that $M$ is totally geodesic.

\textbf{Case 4: $v_3 = (v_1^2+v_2^2) \left[ v_1^2 \sin(2(\theta_1-\theta_2))-(v_1^2+v_2^2)\sin(2(\theta_1-\theta_3)) \right] = 0$.} This case is analogous to Case 3. In fact, if we interchange $E_2$ and $E_3$ and replace $E_1$ by $-E_1$ in order to keep the orientation, $v_2$ will interchange with $v_3$ and $\theta_2$ with $\theta_3$, which reduces Case 4 to Case 3.

\textbf{Case 5.} In the last case, we have
\begin{equation}\label{system2}
\begin{aligned}
& (4 v_1^4 + v_3^4 + 4v_1^2v_2^2 + 12 v_1^2v_3^2 + v_2^2v_3^2) \sin(2 (\theta_1 - \theta_2)) \\
& \hspace{4cm} -(4v_1^4 + 4v_2^4 + 3 v_3^4 + 8v_1^2v_2^2  + 7 v_2^2 v_3^2) \sin(2 (\theta_1 - \theta_3)) = 0, \\
& (4v_1^4 + 3v_2^4  + 4 v_3^4 + 8v_1^2v_3^2 + 7v_2^2v_3^2) \sin(2(\theta_1 - \theta_2)) \\ 
& \hspace{4cm} - (4v_1^4 + v_2^4 + 12v_1^2v_2^2 + 4v_1^2v_3^2 + v_2^2v_3^2) \sin(2(\theta_1-\theta_3)) = 0.
\end{aligned}
\end{equation} 
We regard \eqref{system2} as a system of equations in the unknowns $\sin(2 (\theta_1 - \theta_2))$ and $\sin(2 (\theta_1 - \theta_3))$. Let us first assume that the determinant of the system is non-zero. This implies that $\sin(2(\theta_1-\theta_2))=\sin(2(\theta_1-\theta_3))=0$, which leads to $\theta_1=\theta_2+k_1\frac{\pi}{2}$ and  $\theta_1=\theta_3+k_2\frac{\pi}{2}$ for some integers $k_1$ and $k_2$. These equalities imply that $\theta_2=\theta_3+k_3\frac{\pi}{2}$, with $k_3=k_2-k_1$. We remark that not all $k_i$ can be odd. At least one of them is even, which implies that at least two of the angle functions are equal modulo $\pi$ and hence that the submanifold is totally geodesic by Lemma \ref{lem:angle_functions}. Finally, we assume that the determinant of the system is zero, i.e., that
\begin{equation*}
4 (v_2^2 + v_3^2) (v_1^2 + v_2^2 + v_3^2) (2 v_1^2 + v_2^2 + v_3^2) (4 v_1^2 - 3 (v_2^2 + v_3^2))=0.
\end{equation*}
Given \eqref{ec}, this implies that $v_2=v_3=0$, which was already investigated in Case 1. 

We conclude that, in all of the cases, we are dealing with a totally geodesic submanifold, which finishes the proof.
\end{proof}

\end{document}